\documentclass[12pt]{amsart}
\usepackage{amsmath,amssymb,amsthm,xspace,a4wide,amscd}
\usepackage[all]{xy}
\newtheorem{example}{Example}
\newtheorem{remark}{Remark}
\newtheorem{theorem}{Theorem}
\newtheorem{lemma}{Lemma}
\newtheorem{proposition}{Proposition}
\newtheorem{corollary}{Corollary}

\newcommand{\C}{\ensuremath{\mathbb{C}}\xspace}

\renewcommand{\P}{\ensuremath{\mathcal{P}}\xspace}

\newcommand{\q}{\ensuremath{\mathcal{Q}}}

\renewcommand{\L}{\ensuremath{\Lambda}}
\newcommand{\g}{\ensuremath{\mathfrak{g}}}
\newcommand{\G}{\ensuremath{\mathfrak{G}}}
\renewcommand{\H}{\ensuremath{\mathfrak{H}}}
\newcommand{\N}{\ensuremath{\mathfrak{N}}}
\newcommand{\m}{\ensuremath{\mathfrak{m}}}
\newcommand{\p}{\ensuremath{\mathfrak{p}}}
\newcommand{\B}{\ensuremath{\mathfrak{B}}}
\newcommand{\n}{\ensuremath{\mathfrak{N}}}
\newcommand{\Z}{\ensuremath{\mathbb{Z}}\xspace}

\newcommand{\htt}{\operatorname{ht}\xspace}
\newcommand{\Spec}{\operatorname{Spec}\xspace}

\renewcommand{\phi}{\varphi}
\def\D{\Delta}

\begin{document}
\title{Structure of parabolically induced modules for  affine Kac-Moody algebras}
\author{Vyacheslav Futorny}
\author{Iryna Kashuba}

\address{ Institute of Mathematics, University of
S\~ao Paulo, Caixa Postal 66281 CEP 05314-970, S\~ao Paulo,
Brazil}\email{futorny@ime.usp.br} \email{kashuba@ime.usp.br}
\maketitle
\centerline{Dedicated to Efim Zelmanov in the occasion of his 60th birthday}

\begin{abstract}
The main result of the paper establishes the irreducibility of a large family of nonzero central charge induced modules 
over Affine Lie algebras for any non standard parabolic subalgebra. It generalizes all previously known partial results and provides a 
a construction of many new irreducible modules.

\end{abstract}

\section{Introduction}\label{s1}
Let $\G$ be an  affine
Kac-Moody algebra with
 a $1$-dimensional center
$Z=\C c$ and  a fixed Cartan subalgebra. 

The main problem in the representation theory of  affine Kac-Moody algebras is a classification of all irreducible weight representations.
Such classification is known in various subcategories of weight modules, e.g. in the category $\mathcal O$, in its generalizations \cite{F1}, \cite{C1}, \cite{FS},  
in the category of modules with finite dimensional weight multiplicities and nonzero central charge
\cite{FT}.  An important tool in the construction of representations of  affine Lie algebras
is a parabolic induction. The conjecture (\cite{F2}, Conjecture 8.1)
indicates that induced modules are construction devices for
irreducible weight modules. This conjecture is known to be true for
$A_1^{(1)}$ (\cite{F4}, Proposition 6.3), for $A_2^{(2)}$, \cite{Bu}  and for all  affine Lie
algebras in the case of modules with finite-dimensional weight
spaces  \cite{FT}, \cite{DG}.

Simplest case of parabolic induction corresponds to the induction from Borel subalgebras.  Standard
 examples of Borel subalgebras arise from taking  partitions of
the root system.   For  affine
algebras there is always a finite number of conjugacy
classes by the Weyl  group of such partitions and corresponding Borel subalgebras. 
Verma type modules induced from these  Borel subalgebras were first studied and classified by Jakobsen
and Kac \cite{JK1, JK2}, and by Futorny \cite{F1, F3}, and were further developed in \cite{C1}, \cite{FS}, \cite{F2}, \cite{F4} and
references therein. We will consider a more  general definition of a Borel subalgebra (see below).

Nontrivial (different from Borel) parabolic subalgebras are divided into two groups, those with finite dimensional Levi subalgebras and those with infinite dimensional one. In this paper we are interested in the second case.  The simplest non trivial example is given by a  parabolic subalgebra whose  Levi factor is the Heisenberg subalgebra together with Cartan subalgebra. 
Corresponding induced modules were studied in  recent papers \cite{FK2} and \cite{BBFK}.  It was shown that any irreducible $\Z$-graded module 
 over the Heisenberg subalgebra with a nonzero central charge induces  
the irreducible $\G$-module. 
In \cite{FK1} a similar reduction theorem was shown for pseudo parabolic subalgebras. These parabolic subalgebras give a particular class of non-solvable parabolic subalgebra  of $\G$ with infinite dimensional 
Levi factor. The main results of \cite{FK1} states that in this case the parabolic induction preserves irreducibility if the central charge is nonzero.  The technique used in the proofs in \cite{FK1} and \cite{FK2} are different and somewhat complementary.

The main purpose of the present paper is to show that in the
affine setting  both these cases of parabolic induction (and hence all known cases) can be extended to a more general result for modules
with nonzero central charge. 

 For any Lie algebra $\mathfrak a$ we denote by $U(\mathfrak a)$  the universal enveloping algebra of  $\mathfrak a$.




Denote by $G$ the Heisenberg subalgebra of $\G$ generated by all imaginary root subspaces of $\G$.
Let $\P\subset \G$ be a  parabolic subalgebra of $\G$ such that $\P=\mathfrak l\oplus \mathfrak n$ is a Levi decomposition and $\mathfrak l$ is an infinite dimensional Levi factor.  Denote by $\mathfrak l^0$ the Lie subalgebra of $\mathfrak l$ generated by all its real root subspaces and $\H$. Let $G({\mathfrak l})$ be a subalgebra of $\mathfrak l^0$ spanned by its imaginary root subspaces.
Then $\mathfrak l=\mathfrak l^0+ G_{\mathfrak l}$ where  $G_{\mathfrak l}\subset G$ is the orthogonal complement of  $G({\mathfrak l})$ in $G$ with respect to the Killing form, that is 
$G=G({\mathfrak l})+G_{\mathfrak l}$, $[G_{\mathfrak l}, \mathfrak l^0]=0$ and  $\mathfrak l^0\cap G_{\mathfrak l}=\mathbb C c$.

For a Lie algebra $\mathfrak{a}$  containing the Cartan subalgebra $\H$ we say that a module $V$ is a {\em weight} module if $V=\oplus_{\mu\in \H^*}
V_{\mu}$, where $$V_{\mu}=\{v\in V|hv=\mu(h)v, \forall h\in \H\}.$$ 
We denote by $\mathcal W_{\G}$ (respectively, $\mathcal{W}_{\mathfrak l}$) the category of weight $\G$-modules (respectively, $\mathfrak l$-modules) with respect to the common Cartan subalgebra $\H$ of both $\G$ and $\mathfrak l$. We say that a module $V$ from either category has a nonzero central charge if the central element of $\G$ acts on $V$ as a nonzero scalar.  If $N\in  \mathcal{W}_{\mathfrak l}$ then denote by ${\rm ind}_{N}(\P, \G)$ the induced $\G$-module $U(\G)\otimes_{U(\P)}N$, where $\mathfrak n N=0$. This defines a functor ${\rm ind}(\P, \G)$
from the category  $\mathcal{W}_{\mathfrak l}$  to the category  $\mathcal W_{\G}$. Denote by $\widetilde{\mathcal{W}}_{\mathfrak l}$ the full subcategory of $\mathcal{W}_{\mathfrak l}$ consisting of those modules on which the central element $c$ acts injectively and  
let $\widetilde{{\rm ind}}(\P, \G)$ be the restriction of ${\rm ind}(\P, \G)$ onto $\widetilde{\mathcal{W}}_{\mathfrak l}$. 

Since $\mathfrak l$ is a sum of two commuting Lie subalgebras  $\mathfrak l^0$ and $G_{\mathfrak l}$ then a natural way to construct irreducible modules in 
$\widetilde{\mathcal{W}}_{\mathfrak l}$ is to take a tensor product of an irreducible weight module $L$ over $\mathfrak l^0$ with a $\mathbb Z$-graded irreducible module $T$ over 
$G_{\mathfrak l}$ with the same scalar action of $c$. We will call such modules {\em{tensor}}.  

For any positive integer $k$, denote $\G_k=\G_{k\delta}\oplus \mathbb C c\oplus \G_{-k\delta}$ (see notations in the next section). We say that a $\G_k$-module $S$ is 
$U(\G_{k\delta})$-surjective (respectively $U(\G_{-k\delta})$-surjective) if for any two elements $s_1, s_2\in S$ there exist $s\in S$ and  $u_1, u_2\in U(\G_{k\delta})$ (respectively, $u_1, u_2\in U(\G_{-k\delta})$) such that  $s_i=u_is$, $i=1,2$. 
A 
$G_{\mathfrak l}$-module $T$ is admissible if for any positive integer $k$,  any  cyclic $\G_k$-submodule $T'\subset T$ is $U(\G_{k\delta})$-surjective or $U(\G_{-k\delta})$-surjective.

A tensor module $L\otimes T$ is called {\em{admissible}} if $T$ is an admissible $G_{\mathfrak l}$-module. 
All known to us examples of irreducible  $\mathfrak l$-modules are admissible tensor modules. 
On the other hand we do not have sufficient evidence to expect that admissible tensor modules  exhaust all irreducible modules in $\widetilde{\mathcal{W}}_{\mathfrak l}$  or even tensor modules.

Our main result is the following theorem:

\begin{theorem}\label{the-main} Let $\P\subset \G$ be a parabolic subalgebra of $\G$ such that $\P=\mathfrak l \oplus \mathfrak n$ is a Levi
decomposition and $\mathfrak l$ is infinite dimensional Levi factor. Then  $\widetilde{{\rm ind}}_N(\P, \G)$ is an irreducible
$\G$-module for any irreducible admissible tensor module $N$ from $\widetilde{\mathcal{W}}_{\mathfrak l}$.
 \end{theorem}

We see that this result is quite general allowing one to induce from an arbitrary irreducible admissible tensor 
 $\mathfrak l$-module   with nonzero central charge and to construct many new irreducible modules over $\G$. 
This generalises the results from \cite{FK1}, \cite{FK2}, \cite{BBFK}.
In fact, it is possible to go beyond the category of weight modules but some grading is required in order to apply the same technique.   
Note that the proof of Theorem 1 in \cite{FK2} is only valid for admissible $G$-modules, it is a particular case of the theorem above.

 All results in
the paper hold for both {\em untwisted} and {\em twisted}  affine
Lie algebras.

\section{Preliminaries}

 We address to \cite{K}
for the basics of the Kac-Moody theory.
The  affine Lie algebra $\G$ has the  root
decomposition
$$
\G = \H \oplus (\oplus_{\alpha \in \Delta}
\G_\alpha),
$$
where $\G_\alpha = \{ x \in \G \, | \, [h, x] = \alpha(h) x
{\text{ for every }} h \in \H\}$ and $\Delta$ is  the root
system of $\G$. 
Let $\pi$ be a fixed basis of the root system $\D$.
Then the root system $\D$ of  has a natural partition into positive
and negative roots with respect to $\pi$, $\D_+$ and $\D_-$ respectively. 
Let $\delta\in \D_+(\pi)$ be the
indivisible imaginary root.
  Then the set of all imaginary roots is $\Delta^{im}=\{k\delta|
k\in\Z\setminus\{0\}\}$.  Let $G= \oplus_{k \in \Z \setminus \{0\}} \G_{k\delta} \oplus \mathbb C c$, a Heisenberg subalgebra of $\G$. Then $G$ has 
a triangular decomposition $G = G_- \oplus \C c \oplus G_+$, where
$G_{\pm} = \oplus_{k>0} \G_{\pm k \delta}$. 


  Denote by  ${\mathfrak g}$
the underlined simple finite dimensional Lie algebra that has $\{\alpha_1, \ldots, \alpha_N\}\subset \pi$ as set of simple roots and $\dot \D = \dot \D_+ \cup \dot \D_-$ be
a decomposition of the  root system $\dot \D$ of ${\mathfrak g}$ into positive and negative to this set of simple roots.
When there are  roots of two lengths in ${\mathfrak g}$, we  set $\dot \D_l$ and $\dot \D_s$ to be the long and short roots in $\dot \D$ respectively.
 The real  roots  $\D^{\mathsf {re}}$ of $\D$  can be described as follows: $\D^{\mathsf {re}}=S\cup -S$, where
 \begin{equation}\label{eq:posreal}
S = \begin{cases}   \{ \alpha + n\delta\ |\ \alpha \in \dot \D_+, \, n \in \Z\},
 \hspace{1.7 truein}  \hbox{\rm  if $r = 1$  (the untwisted case),  }  \\
 \{ \alpha + n\delta\ |\ \alpha \in (\dot \D_s)_+, n \in \Z\}\, \cup \,  \{ \alpha + nr\delta\ |\ \alpha \in (\dot \D_l)_+, n \in \Z \}
  \qquad \hbox{\rm  if $r = 2,3$ and} \\
  \hspace{4.7 truein} \hbox{\rm  not $\mathsf{A}_{2\ell}^{(2)}$ type,}  \\
 \{ \alpha + n\delta\ |\ \alpha \in (\dot \D_s)_+, n \in \Z\}\, \cup \,  \{ \alpha + 2n\delta\ |\ \alpha \in (\dot \D_l)_+, n \in \Z \} \\
 \hspace{1.2 truein} \cup \, \{\frac{1}{2}\left ( \alpha + (2n-1)\delta\right) \ |\ \alpha \in (\dot \D_l)_+, n \in \Z \} \qquad  \qquad \hbox{\rm  if $\mathsf{A}_{2\ell}^{(2)}$ type.}  \\
\end{cases}
\end{equation}

\section{Parabolic induction}

In this section we consider our main tool of constructing new modules - parabolic induction.
We start with the discussion of the Borel subalgebras.

\subsection{Borel subalgebras}

A subalgebra $\B\subset \G$ is called a {\em Borel subalgebra} if it contains $\H$ (and hence has a root decomposition) and there exists an automorphism $\sigma$ of $\G$ satisfying 

\begin{itemize}
\item $\sigma(\H)=\H$;
\item $\B+\sigma(\B)=\G$;
\item $\B\cap\sigma(\B)=\H$.
\end{itemize}

Note that this definition is more general than the usual definition of Borel subalgebras associated with closed partitions of root systems \cite{F1}, \cite{DFG}. 

 Consider a  subset $P\subset \D$ such that 
 $P\cap (-P)=\emptyset$ and $P\cup
(-P)=\D$.  Denote by $\B_P$   a Lie  subalgebra of $\G$
generated by $\H$ and the root spaces
${\mathfrak G}_{\alpha}$ with $\alpha \in P$.  We will say that $P$ is
 a {\em quasi partition} of $\D$ if for any root $\alpha$ of $\B_P$ we have $\alpha\in P$. 
  Note that in a contrast with a usual concept of a {\em partition} of the root systems \cite{F1}
  we do not require $P$ to be a closed subset with respect to the sum of roots, that is whenever $\alpha$
 and $\beta$
are in $P$ and $\alpha + \beta$ is a root, then $\alpha + \beta
\in P$.  A subset of  all real roots of any quase partition is  closed with respect to the sum of the roots while it is not necessary for imaginary roots (cf. \cite{BBFK}).  Clearly,  $\B_P$ is a Borel subalgebra for any quase partition $P$. These subalgebras are a main source of examples of Borel subalgebras, though they do not exhaust all of them as shown in the following remark.

\begin{remark}\label{imag-multiple}
In each subspace $G_{k\delta}$ we can choose a commuting basis $x_{k1}, \ldots, x_{k s_k}$ such that 
  $[x_{ki}, x_{-kj}]=\delta_{ij}$ for any $k$ and all $i,j$. 
We can define a triangular decomposition of $G$: $G=G_1\oplus \mathbb C c\oplus G_{-1}$, where  for each nonzero $k$,  $G_{k\delta}\subset G_i$ implies    $G_{-k\delta}\subset G_{-i}$, $i=1, -1$.
 On the other hand we can obtain a more general triangular decomposition
    by splitting for every $k$,  $x_{kj}$ and $x_{-kj}$ between $G_1$ and $G_{-1}$ for each $j$ independently, $j=1, \ldots, s_k$. Each such triangular decomposition can be extended to the following decomposition of $\G$. Denote by $B_{\pm}$ a Lie subalgebra of $\G$ generated by all root spaces $\G_{\pm\beta}$, 
 $\beta\in \{\alpha+k\delta|\alpha \in \dot{\Delta}_+, k\in \Z\}$ and $G_{\pm 1}$. Then 
 $$\G=B_-\oplus \H\oplus B_+$$ and   $ B=\H\oplus B_+$ is a Borel subalgebra of $\G$. In particular,  if $x_{kj}$ are in the same $G_i$ 
 for all positive $k$ and all $j=1, \ldots, s_k$ then $B$ corresponds to a partition of the root system. If for some positive $k\neq m$, $x_{kj}$ and 
  $x_{m r}$  belong to different  $G_i$ but for any $k$,  $x_{kj}$ are in the same $G_i$ for all $j=1, \ldots, s_k$, then $B$ corresponds to a quase partition of the root system (cf. $\B_{\rm nat}^{\phi}$ below). 
\end{remark}

Classification of partitions of the root system and corresponding Borel subalgebras was obtained  in \cite{JK1} and \cite{F1}.  Classification of quasi partitions follows from \cite{BBFK}. Finally, the classification of all Borel subalgebras defined above can easily be deduced from Remark \ref{imag-multiple}.    

There are two extreme Borel subalgebras, the {\rm standard} Borel subalgebra which corresponds to the partition $P=\Delta_+$ and
 the {\em natural} Borel subalgebra $\B_{\rm nat}=\B_P$ which corresponds to the partition 
$$P_{\rm nat}=\{\alpha+k\delta|\alpha \in \dot{\Delta}_+, k\in \Z\}\cup \{n\delta| n> 0\}.$$
These Borel subalgebras are not conjugated by the Weyl group. For other conjugacy classes of Borel subalgebras by the Weyl group see \cite{F1}.    
We will be interested mainly in  $\B_{\rm nat}$ in this paper.  Starting from this Borel subalgebra one can construct a family of twisted subalgebras as in \cite{BBFK}. 
For a function $\phi: \mathbb{N} \rightarrow \{\pm \}$, set
 $$
P^{\phi}= \{\alpha+k\delta|\alpha \in \dot{\Delta}_+, k\in \Z\}
\cup \{n \delta\ |\  n \in \mathbb N, \phi(n)=+\}\cup \{-m \delta\
|\ m \in \mathbb N, \phi(m)=-\}
$$
and $\B_{\rm nat}^{\phi}=\B_{P^{\phi}}$.  Further examples of Borel subalgebras can be obtained by combining $\phi$-twisting with the procedure described in Remark~\ref{imag-multiple}.

\subsection{Parabolic subalgebras}

A subalgebra $\P\subset \G$ is called a {\em parabolic subalgebra} if it contains a Borel subalgebra.  There are essentially two types of parabolic subalgebras: those containing the standard Borel and those containing one of the twisted Borel subalgebras $\B_{\rm nat}^{\phi}$.  We call them {\em type I} and {\em type II} parabolic subalgebras respectively (cf. \cite{F2} for details). We address in the paper the parabolic subalgebras of {\em type II}. Even though all the results of this paper are valid for all parabolic subalgebras of {\em type II} we will assume for simplicity that a fixed parabolic subalgebra $\P$ of {\em type II} contains $\B_{\rm nat}$. We will describe all such $\P$'s which contain properly $\B_{\rm nat}$ (cf. also Proposition 3.3, \cite{FK1}).

Let $N$ be the rank of the underlined simple finite dimensional Lie algebra $\dot{\mathfrak g}$ and $\pi_0=\{\alpha_1, \ldots, \alpha_N\}\subset \pi$  the set of simple roots of  
$\dot{\mathfrak g}$.  Set   $I = \{1, \dots, N\}$ and choose any proper subset 
 $J \subset I$.  Let $\pi^J = \{ \alpha_j
\in \pi\ |\ j \in J\}$.  
Denote by $\dot \D^J$  the finite root system generated by the 
roots in $\pi^J$ ( $\dot \D^J = \emptyset$  if $J
= \emptyset$).  Now consider the affinization
$$
\D^J = \{ \alpha + n\delta \in \D\ |\ \alpha \in \dot \D^J, n \in \Z\} \cup
\{n\delta \ |\ n \in \Z \setminus \{0\} \},
$$
of $\dot \D^J$ (in $\D$).

Define $\P_J=\B_{\rm nat}+\sum_{\alpha\in \D^J}\G_{\alpha}$. This is the parabolic subalgebra of type II associated with $J$. 
Again it admits modifications as in Remark~\ref{imag-multiple} but we will not consider it though the main statement remains valid also in this case.  If $J
= \emptyset$ then $\P_{\emptyset}=\B_{\rm nat}+G$. 
The parabolic subalgebra
$\P_J$ has the  Levi decomposition
$\P_J=\mathfrak l_J\oplus \mathfrak n_J$ where $\mathfrak l_J$ is a Lie subalgebra generated by $\G_{\alpha}$ with $\alpha\in \D^J$ and 
$\mathfrak n_J=\sum_{\alpha\in P_{\rm nat}\setminus \D^J}\G_{\alpha}$.  Note that $\mathfrak l_{\emptyset}=G+\H$.

Denote $\mathfrak n_{\bar J}=\sum_{\alpha\in -P_{\rm nat}\setminus \D^J}\G_{\alpha}$. Then we have the following decomposition of $\G$:
$\G=\mathfrak n_{\bar J}\oplus \P_J= \mathfrak n_{\bar J}\oplus \mathfrak l_J\oplus \mathfrak n_J$. Note that this may not be a triangular decomposition in the classical
sense of \cite{MP}. Nevertheless this decomposition allows us to construct families of induced modules.

 We say that 
a subset $S \subseteq J$ is connected if the Coxeter-Dynkin diagram
associated to the simple roots $\alpha_i$, $i\in S$ is connected.  Then
$J=\cup_{t\in T}S_t$ where $S_t$'s are connected components of the
Coxeter-Dynkin diagram associated to $J$.  Each subset $S_t$ gives rise to an affine root subsystem of 
$\D^J $ which generates an
 affine Lie subalgebra $\mathfrak l_J(S_t)\subset \mathfrak l_J$.  Also denote by $G_J\subset
G$  the orthogonal completion (with respect to the Killing form)
of the Heisenberg subalgebra  $G(\mathfrak l_J)$ of $\mathfrak l_J$. Then $G=G_J+G(\mathfrak l_J)$,
$[G_J, \mathfrak l_J]=0$ and we have $$\mathfrak l_J=\sum_{t\in T}\mathfrak l_J(S_t) + G_J +\H.$$

\subsection{Induced modules}

We will assume that  $\P_J$ is the parabolic  subalgebra of type II associated with fixed subset $J$, $\P_J=\mathfrak l_J\oplus \mathfrak n_J$.
 Let $N$ be a weight (with respect to $\H$) module over  $\P_J$   with a
trivial action of $\mathfrak n_J$.
Define the induced $\G$-module
$$M_{J}(N)={\rm ind}_N(\P_J, \G).$$

This is the {\em generalized Imaginary Verma module} associated with $J$ and $N$. 

  If $N$ is irreducible then  $M_{J}(N)$  has a
 unique irreducible quotient
$L_{J}(N)$. Basis properties of  modules $M_J(N)$ are collected in the following proposition.  Details of the proofs can be
found in \cite{F2}.

\begin{proposition} Let  $J \subseteq 
I$ and $N$ irreducible weight $\mathfrak l_J$-module.  Then $M_J(N)$ has the following properties. \begin{itemize}
\item[(i)] The module $M_J(N)$ is a free
$U(\frak{n}_{\bar J})$-module.
 \item[(ii)]  Let $V$ be a nonzero
$\G$-module generated by a  weight vector $v$ such that $\frak{n}_{J}v=0$.  Set $N=U(\mathfrak l_J)v$. Then 
 there exists a unique surjective homomorphism
$\psi:M_J(N) \mapsto V$ such that $\psi (1\otimes v) = v$. If $V$ is irreducible then 
  $N$ is irreducible $\mathfrak l_J$-module and $V\simeq L_J(N)$.
 \item[(iii)] $M_J(N)$ is a weight module. Moreover, $0 < \dim M_J(N)_{\mu} <
\infty$ if and only if $\mu$ is a weight of $N$ and $0 < \dim  N_{\mu}< \infty$.
\item[(iv)] If $N$ is irreducible $\mathfrak l$-module then the subspace of $\mathfrak n_J$-invariants in $L_J(N)$ is the class of $1\otimes N$.
\item[(v)] If $J=\emptyset$ and $N$ is a highest weight $G$-module generated by a weight vector $v$ such that  $hv=\lambda(h)v$ for any $h\in \H$ and some $\lambda\in \H^*$ 
and $\mathfrak n_{\emptyset} v=0$ then 
$M_{\emptyset}(N)$ is an Imaginary Verma type module $M(\lambda)$ generated by an eigenvector for the natural Borel subalgebra $\B_{\rm nat}$.
\end{itemize}
\end{proposition}

\begin{remark}
\begin{itemize}
 \item[(i)] Generalized loop modules considered in \cite{FK2} and, in particular,  $\phi$-Imaginary Verma modules \cite{BBFK} are partial cases of modules $M_{\emptyset}(N)$
  when $N$ is irreducible $G$-module;
   \item[(ii)] Pseudo-parabolic induction considered in \cite{FK1} is a particular case of  modules $M_{J}(N)$ where   $ G_J N=0$. 
\end{itemize}
\end{remark}

 Set $M^t_J(N) := 1\otimes N$. This is the "top" part of 
  $M^t_J(N)$  which generates $M_{J}(N)$.

\subsubsection{Tensor $\mathfrak l_J$-modules}

Denote $\mathfrak l_J^0=\sum_{t\in T}\mathfrak l_J(S_t)+\H$. Hence $\mathfrak l_J=\mathfrak l_J^0+ G_J$ and $[\mathfrak l_J^0, G_J]=0$.  If $V$ is an irreducible weight 
$\mathfrak l_J^0$-module and $W$ is an irreducible $\mathbb Z$-graded $G_J$-module with the same central charge then $V\otimes W$ is naturally an 
$\mathfrak l_J$-module, a tensor module.  If the central charge is $a\in \mathbb C$ then $V\otimes W$ is a module over the tensor product $U(\mathfrak l_J^0)/(c-a)\otimes U(G_J)/(c-a)$. 

In the extreme case when $J=\emptyset$ we have $\mathfrak l_J^0=\H$ and  $G_J=G$. Hence $V$ is a $1$-dimensional space and $W$ is a $\mathbb Z$-graded $G$-module. 
Parabolic induction functor from $G$ to $\G$ was considered in \cite{FK2}. 

Suppose now that $J\neq \emptyset$. Let $N$ be an irreducible weight $\mathfrak l_J$-module with a nonzero central charge $a$. Assume that for each positive integer $k$ 
either $(G_J)_{k\delta}$ or  $(G_J)_{-k\delta}$ acts trivially on $N$. Such $\mathfrak l_J$-modules were considered in \cite{FK1} when inducing from pseudo parabolic subalgebras. 
We will show that any such irreducible module is a tensor module. Without loss of generality we assume that $(G_J)_{k\delta}N=0$ for all positive integer $k$.

\begin{proposition} 
 $N$ is a tensor $\mathfrak l_J$-module, that is $N\simeq V\otimes W$ where $V$ is irreducible weight $\mathfrak l_J^0$-module and $W$ is irreducible 
 $\mathbb Z$-graded $G_J$-module.
\end{proposition}

\begin{proof}
Choose nonzero element $v\in N$. Denote $V=U(\mathfrak l_J^0)v$ and $W=U(G_J)v$.  It is standard that $W$ is irreducible  $G_J$-module since the central charge is nonzero 
and $G_J^+v=0$. Suppose $V$ is not irreducible and $V'$ is a nonzero proper $\mathfrak l_J^0$-submodule. Take any nonzero $v'\in V'$. Then $v'=xv$ for some $x\in U(\mathfrak l_J^0)$. 
Since $N$ is irreducible there exists 
$y\in U(G_J)$ such that $yxv=v$. Moreover, we can assume that $y\in U(G_J^-)$. But the $G_J$-module $U(G_J)xv$ is irreducible. Hence, there exists $y'\in U(G_J^+)$ such that 
$y'yxv=y'v=xv$. But $y'v=0$, thus $v'=xv=0$ which is a contradiction. We conclude that $V$ is $\mathfrak l_J^0$-module. Then $V\otimes W$  is irreducible $\mathfrak l_J$-module. 
Consider a map $f:V\otimes W\rightarrow N$, sending $xv\otimes yv$ to $xyv$ which is clearly a homomorphism since $\mathfrak l_J^0$ and $G_J$ commute.
 We immediately see that $f$ is surjective since $N$ is irreducible. If $xyv=0$ then choose $y'$ as above. We have $0=y'yxv=xv$ and $f$ is injective. Therefore, $N\simeq V\otimes W$.
\end{proof}

These are all known cases when parabolic induction preserves irreducibility and in all cases we induce from certain tensor modules. Combining  techniques from  
\cite{FK1} and \cite{FK2} we will extend the proof to all tensor modules.

\section{Irreducibility of Generalized Imaginary Verma modules}
In this section we prove our main result by finding conditions for a module $M_{J}(N)$ to be irreducible. 

Let 
$\P_J=\mathfrak l_J\oplus \mathfrak n_J$ and $N\in W_{\mathfrak l}$. Denote by $T_J(N)$ the subspace of $\mathfrak n_J$-invariants in $M_J(N)$, that is $v\in T_J(N)$ 
if and only if $ \mathfrak n_J v=0$.

\begin{theorem}\label{thm-glav}
If $U$ is an irreducible admissible tensor module in $\widetilde{W}_{\mathfrak l}$ then
$T_J(U)=M^t_J(U)$. 
\end{theorem}

The proof of Theorem \ref{thm-glav} combines the proofs of Theorem 2 from \cite{FK2} and Theorem 3.1 from \cite{FK1} where particular cases of parabolic induction were considered.

For any
subset $\omega\subset I$, let $Q^{\omega}_{\pm}$ denote a
semigroup of $\H^*$ generated by $\pm \alpha_i$, $i\in \omega$. Set   
$Q^J = \oplus_{j \in J} \Z \alpha_j
\oplus \Z\delta$ and $Q_{\pm}=Q^{I}_{\pm}$. Then $Q_{\pm}^J= Q^J \cap Q_{\pm}$.

Let $\alpha\in Q^{J}_{-}$
and $\alpha = -\sum_{j\in \omega} k_j \alpha_j$,
where each $k_j$ is in $\mathbb Z_{\geq 0}$. We set
$\mathsf {ht}_{J}(\alpha) = \sum_{j=1}^n k_j$, the \emph{J-height} of $\alpha$.

Let  $v\in M_J(N)$ be a
nonzero weight element.  Then $$v=\sum_{i\in R} u_iv_i,$$ for some finite set $R$, 
where $u_i\in U(\mathfrak n_{\bar{J}})$ are linearly independent homogeneous elements, 
$v_i\in N$, $i\in R$.  Since $v$ is a weight vector then each $u_i$ is a homogeneous element of $U(\mathfrak n_{\bar{J}})$. Its homogeneous degree is an element 
of $Q_-^{\bar{J}}+ Q^J+\Z \delta$. Suppose that $u_i$ has homogeneous degree  
 $$\phi_i= -\sum_{j\in \bar{J}} k_{ij} \alpha_j+\sum_{j\in J} l_{ij} \alpha_j+ m_i\delta,$$ where $k_{ij}\in\mathbb Z_{\geq 0}$ not all zeros, $l_{ij}\in\mathbb Z$ and $m_i\in \Z$.   
 Since $v$ is a weight vector then all $\phi_i$ have the same J-height. We will call it the {\em J-height} of $v$ and denote $\mathsf {ht}_{J}(v)$.

 \begin{lemma}\label{lem-ht-big}
 Suppose  $U\in \widetilde{W}_{\mathfrak l}$ and $v\in M_J(U)$  a
nonzero weight element such that $\mathsf {ht}_{J}(v)>1$. Then there exists $u\in U(\mathfrak n_J)$ such that $uv\neq 0$ and $\mathsf {ht}_{J}(uv)=\mathsf {ht}_{J}(v)-1$.
 
 \end{lemma}

 \begin{proof}
Let $v=\sum_{i\in R} u_iv_i,$
where $u_i\in U(\mathfrak n_{\bar J})$ are linearly independent homogeneous elements and 
$v_i\in U$ are nonzero elements, $i\in R$. We assume that for each $i$, $u_i$ has homogeneous degree  
 $\phi_i$ and all $\phi_i$ have the same J-height.  Then we can apply exactly the same argument as in the proof of the induction step 
 in Lemma 5.3 in \cite{BBFK}. We refer to \cite{BBFK} for details. 
 \end{proof}

 It follows immediately from Lemma \ref{lem-ht-big} that  $T_J(N)$ can not contain  nonzero elements of J-height $\mathsf {ht}_{J}(v)>1$. Indeed,  
 any such element would generated a proper submodule whose elements would have   J-heights $\geq \mathsf {ht}_{J}(v)$ which contradicts Lemma \ref{lem-ht-big}. 
 Note that the proof of Lemma \ref{lem-ht-big} does not work in the case when $\mathsf {ht}_{J}(v)=1$. This case requires a more delicate treatment.
 We consider first the case when $J=\emptyset$. This case was treated in \cite{FK2} where  
 the key point was Lemma 1. But the proof of this lemma is somewhat incomplete so we address the proof here in more details.
 
 Set  $\mathfrak l=\mathfrak l_{\emptyset}=G+\H$.

 \begin{lemma}\label{lem-heis}
 Let   $W\in \widetilde{W}_{\mathfrak l}$ be an irreducible and admissible, $v\in W$ a nonzero element and 
 $u_2, \ldots, u_s\in U(G)$  nonzero homogeneous elements of nonzero degrees $k_2, \ldots k_s$ respectively, that is $u_i\in U(G)_{k_i\delta}$, 
 such that $u_iv\neq 0$, $i=2, \ldots, s$, $k_i\neq k_j$ if $i\neq j$. Then  there exists
  $N\in \mathbb Z$ for which
$$z_N=x_N v+ \sum_{i=2}^s x_{N-k_i} u_i v\neq 0$$ for any choice of nonzero $x_k\in G_{k\delta}$ such that $[x_k, x_{-k}]\neq 0$.
\end{lemma}

 \begin{proof}
 Since $W$ is admissible we can assume without loss of generality  that $k_i+k_j\neq 0$ for all $i\neq j$. This implies $[u_i, u_j]=0$ and $[x_{N-k_i}, x_{-N-k_j}]=0$ for all $i,j$.  
 Fix $N\in \mathbb Z$ and suppose
 that
$z_N=z_{-N}=0.$ We will assume $N$ sufficiently large.
We have $x_{N}v=-\sum_{i=2}^s x_{N-k_i} u_iv$ and $x_{-N}v=-\sum_{i=2}^s x_{-N-k_i}
u_iv$. Then 
$$[x_{N}, x_{-N}]v=(-x_N\sum_{i=2}^s x_{-N-k_i}u_i+x_{-N}\sum_{i=2}^s x_{N-k_i} u_i)v,$$
since $[x_{N}, u_i]=[x_{-N}, u_i]=[x_{N-k_i}, x_{-N}]=[x_{N}, x_{-N-k_i}]=0$.

Fix $j=2, \ldots, s$. Then 
$x_{N-k_j}u_jv=-x_{N}v-\sum_{2\leq i\neq j} x_{N-k_i} u_iv$ and $x_{-N-k_j}u_jv=-x_{-N}v-\sum_{2\leq i\neq j} x_{-N-k_i}
u_iv$. 
We have $$[x_{N-k_j}u_j, x_{-N-k_j}u_j]v= - x_{-N} x_{N-k_j} u_j v- x_{N-k_j}\sum_{2\leq i\neq j} x_{-N-k_i}u_j  u_i  v +$$
$$+  x_{N} x_{-N-k_j}u_j v +  x_{-N-k_j}u_j(\sum_{2\leq i\neq j} x_{N-k_i}u_i) v.$$ Since $k_j\neq 0$, $[x_{N-k_j}u_j, x_{-N-k_j}u_j]=0$ for all $j=2, \ldots, s$.
We sum up all these equalities for $j=2, \ldots, s$:
$$0=\sum_{j=2}^s (- x_{-N} x_{N-k_j} +  x_{N} x_{-N-k_j})u_j v     -   \sum_{j=2}^s   x_{N-k_j}(\sum_{2\leq i\neq j} x_{-N-k_i}u_j  u_i)  v$$
$$+  \sum_{j=2}^s     x_{-N-k_j}(\sum_{2\leq i\neq j} x_{N-k_i}u_ju_i) v  = $$
$$ \sum_{j=2}^s (- x_{-N} x_{N-k_j} +  x_{N} x_{-N-k_j})u_j v =  -[x_{N}, x_{-N}]v\neq 0  $$
which is a contradiction. Hence, $z_N=0$ or $z_{-N}=0$ which completes the proof.

 \end{proof}

  \begin{corollary}\label{cor-ht-empty}
 Suppose $J=\emptyset$,   $W\in \widetilde{W}_{\mathfrak l}$ is irreducible and $v\in M_J(W)$  a
nonzero weight element such that $\mathsf {ht}_{J}(v)=1$. Then there exists $d\in U(\mathfrak n_J)$ such that $dv\neq 0$ and 
$\mathsf {ht}_{J}(dv)=0$.
 \end{corollary}

 \begin{proof}
 Since  $v\in M_J(W)$ a
 weight element then $$v=\sum_{r\in R} d_rw_r,$$
where $d_r\in \mathfrak n_{\bar \emptyset}$ are linearly independent, $w_r\in W$, $r\in R$. Fix $r_0\in R$ and assume $d_{r_0}\in \G_{\phi}$.  Choose an integer 
$N$ and
  let $d\in\G_{-\phi+N\delta}\subset \mathfrak n_{\emptyset}$ be a nonzero element. Then $[d, d_{r_0}]\neq 0$ and we have $$dv=d d_{r_0}w_{r_0} + d \sum_{r\in R, r\neq r_0} d_r w_r=$$
  $$=[d, d_{r_0}]w_{r_0}+  \sum_{r\in R, r\neq r_0}[d, d_r]w_r.$$  Since $W$ is irreducible there exist $u_r\in U(G)$, $r\in R$ such that $w_r=u_rw_{r_0}$. Note that 
  all $u_r$ have different homogeneous degrees as $d_r$ were linearly independent. Note that $[d, d_{r_0}]\in G_{N\delta}$.  Applying Lemma~\ref{lem-heis} 
  for a sufficiently large $N$ we obtain $dv\neq 0$ and $\mathsf {ht}_{J}(dv)=0$. 
 \end{proof}

 Next we consider the case $J\neq \emptyset$. 
 
  \begin{lemma}\label{lem-tensor-module}
 Let $U$ be a tensor module in $\widetilde{W}_{\mathfrak l}$, $U\simeq V\otimes W$ where $V$ is irreducible weight $\mathfrak l_J^0$-module and $W$ is irreducible 
 $\mathbb Z$-graded $G_J$-module. Choose nonzero $v_1,\ldots, v_k\in V$ and nonzero $w_1,\ldots, w_k\in W$. Fix an integer $N>0$ and
   nonzero $x_{\pm N}\in G_{\pm N\delta}$, such that 
 $x_{\pm N}\notin \mathfrak l_J^0$ and $[x_N, x_{-N}]\neq 0$.
  Then for any
 $a_i, b_i\in U(\mathfrak l_J)$, $i=2, \ldots, k$ and sufficiently large $N$ we have 
 $$x_{-N}(v_1\otimes w_1) +\sum_{i=2}^k a_i v_i\otimes w_i\neq 0$$ or 
  $$x_{N}(v_1\otimes w_1) +\sum_{i=2}^k b_i v_i\otimes w_i\neq 0. $$

  \end{lemma}

 \begin{proof}
 Assume
  $$z=x_{-N}(v_1\otimes w_1) +\sum_{i=2}^k a_i v_i\otimes w_i=0.$$
Suppose first that  $[x_{\pm N}, \mathfrak l_J^0]=0$, that is $x_{\pm N}\in G_J$.  Then 
  $$x_{N}z=v_1\otimes x_{N}x_{-N}w_1 +\sum_{i=2}^k a_i v_i\otimes x_{N}w_i=0.$$
  If $N$ is sufficiently large then $x_{N}w_i$ have different gradings than $w_1$, $i=2, \ldots, k$ and thus $x_{N}x_{-N}w_1=0$. Similarly, if 
   $v_1\otimes x_{N}w_1 +\sum_{i=2}^k b_i v_i\otimes w_i=0$ then $x_{-N}x_{N}w_1=0$. But this is a contradiction. 
   Consider now a general case. We have $x_{\pm N}=x_{\pm N}^1+x_{\pm N}^2$ where $[x_{\pm N}^1, \mathfrak l_J^0]=0$ and $x_{\pm N}^2\in \mathfrak l_J^0$.
   Moreover,   $x_{\pm N}^1\neq 0$ since $x_{\pm N}\notin \mathfrak l_J^0$ and $[x_{N}^1,  x_{-N}^1]\neq 0$. 
   Then
  $$z=v_1\otimes x_{-N}^1w_1+ x_{-N}^2v_1\otimes w_1 +\sum_{i=2}^k a_i v_i\otimes w_i=0$$ and 
  $$x_{N}^1z=v_1\otimes x_{N}^1x_{-N}^1w_1+ x_{-N}^2v_1\otimes x_{N}^1w_1 +\sum_{i=2}^k a_i v_i\otimes x_{N}^1w_i=0. $$ Then we proceed as in the previous case and conclude $x_{N}^1x_{-N}^1w_1=0$. Replacing $N$ by $-N$ we obtain $x_{-N}^1x_{N}^1w_1=0$ implying $w_1=0$ which is a contradiction. This completes the proof. 
 \end{proof}

  \begin{corollary}\label{cor-ht-not-empty}
  Let $U\simeq V\otimes W$ be a tensor module in $\widetilde{W}_{\mathfrak l}$,  where $V$ is irreducible weight $\mathfrak l_J^0$-module and $W$ is irreducible 
  $G_J$-module. For a
nonzero weight element  $v\in M_J(U)$ with $\mathsf {ht}_{J}(v)=1$  there exists $u\in U(\mathfrak n_J)$ such that $uv\neq 0$ and 
$\mathsf {ht}_{J}(uv)=0$.
 \end{corollary}
 
 \begin{proof}
Let  $v=\sum_{r\in R} d_r (v_r\otimes w_r),$ where 
 $d_r\in \mathfrak n_{\bar J}$ are linearly independent, $v_r\in V$, $w_r\in W$, $r\in R$. We proceed as in the proof of Corollary~\ref{cor-ht-empty}. 
 Fix $r_0\in R$ and assume $d_{r_0}\in \G_{\phi}$.  Choose an integer 
$N$ and
  let $t_{\pm N}\in\G_{-\phi \pm N\delta}\subset \mathfrak n_{J}$ be  nonzero elements. Then $x_{\pm N}=[t_{\pm N}, d_{r_0}]\neq 0$.
  Moreover, $[x_N, x_{-N}]\neq 0$ and 
  we have $$t_{\pm N}v=t_{\pm N} d_{r_0}(v_{r_0}\otimes w_{r_0}) + t_{\pm N} \sum_{r\in R, r\neq r_0} d_r(v_r \otimes w_r)=$$
  $$=x_{\pm N}(v_{r_0}\otimes w_{r_0}) +  \sum_{r\in R, r\neq r_0}[t_{\pm N}, d_r](v_r \otimes w_r).$$ Applying Lemma~\ref{lem-tensor-module} 
  we conclude that  at least one of $t_{\pm N}v$ is not zero. Since  $\mathsf {ht}_{J}(t_{\pm N}v)=0$ the corollary is proved.
 \end{proof}

 We can now prove  Theorem \ref{thm-glav}. 
 
 \subsection{Proof of Theorem \ref{thm-glav}} Let 
 $v\in T_J(U)$ be a nonzero weight element and $\mathsf {ht}_{J}(v)=s\geq 1$. Consider a $\G$-submodule $\mathcal N$ of $M_J(U)$ generated by 
 $v$, $\mathcal N=U(\G)v$. Since $\mathfrak n_J v=0$ we have that  $\mathcal N$ is a proper submodule of $M_J(U)$ and all its weight elements have 
 $\mathsf {ht}_{J}$ greater or equal than $s$. But this is a contradiction since by
  Lemma~\ref{lem-ht-big}, Corollary~\ref{cor-ht-not-empty} and Corollary~\ref{cor-ht-empty} we can always find $u\in (\G)$ such that 
  $uv\neq 0$ and  $\mathsf {ht}_{J}(uv)=\mathsf {ht}_{J}(v)-1$. Therefore $\mathsf {ht}_{J}(v)$ must be zero and $T_J(U)=M^t_J(U)$.
 
 Applying Theorem \ref{thm-glav} we have 
 
   \begin{corollary}\label{cor-irr}
   If $U$ is an irreducible  admissible tensor module in $\widetilde{W}_{\mathfrak l}$ then the induced module
$M_J(U)$ is irreducible.
   
   \end{corollary}

 Corollary~\ref{cor-irr} immediately implies Theorem~\ref{the-main} which provides a powerful tool to construct new irreducible representations for  affine Lie algebras by inducing from 
 irreducible tensor modules from $\widetilde{W}_{\mathfrak l}$. We conclude with the following observation.

 \begin{remark} 
 It would be interesting to see if Theorem \ref{thm-glav} extends to any admissible  tensor module $U$ in $\widetilde{W}_{\mathfrak l}$. This would lead to an equivalence of  certain subcategories of $\G$-modules and $\mathfrak l_J$-modules. Another problem is to check if  admissible tensor modules exhaust all irreducible modules in 
 $\widetilde{W}_{\mathfrak l}$. 
 \end{remark}

\section{Acknowledgment}
V.F. is
supported in part by the CNPq grant (301320/2013-6) and by the
Fapesp grant (2014/09310-5). I.K. is supported by the CNPq grant (309742/2013-7)   and by the
Fapesp grant (2016/08740-1). The authors are grateful to the referee for useful remarks.

\end{document}